\documentclass[1pt]{article}

\usepackage{amsmath, amssymb, amsthm, amsfonts, cases}
\usepackage{mathrsfs}
\usepackage{url}
\usepackage{authblk}
\usepackage[usenames]{color}
\usepackage{geometry}

\usepackage[english]{babel}
\usepackage{authblk}

\newtheorem{theorem}{Theorem}
\newtheorem{lemma}{Lemma}

\newtheorem{definition}{Definition}

\newtheorem{problem}{Problem}
\newtheorem{assumption}{Assumption}

\newcommand{\be}{\begin{equation}}
\newcommand{\ee}{\end{equation}}
\newcommand{\beq}{\begin{eqnarray}}
\newcommand{\eeq}{\end{eqnarray}}

\newcommand{\ced}{\end{proof}}

\def\first {\vtop{\baselineskip=11pt
		\hbox to 125truept{\hss Yuchao Dong\hss\footnote{  } }}}
\def\second {\vtop{\baselineskip=11pt
		\hbox to 125truept{\hss Qinxin Meng\hss\footnote{  } }}}
\begin{document}
	\title{cument}
\title {Second-Order Necessary Conditions for Optimal Control with Recursive Utilities
\thanks{The first auther gratefully acknowledges finincial support from R\'egion Pays de la Loire throught the grant PANORisk. The second auther was supported by the Natural Science Foundation of Zhejiang Province
for Distinguished Young Scholar  (No.LR15A010001) and the National Natural Science Foundation of China (No.11471079, 11301177)}
}
\date{}

\author[a,b]{Yuchao Dong}
\author[c]{Qingxin Meng\footnote{Corresponding author.
\authorcr
\indent E-mail address: ycdong@fudan.edu.cn(Y. Dong), mqx@hutc.zj.cn(Q. Meng))}}

\affil[a]{\small{School of Mathematical Sciences, Fudan University, Shanghai 200433, China}}
\affil[b]{\small{LAREMA, D\'epartment de Math\'ematiques, Universit\'e d'Angers,2 Bd Lavoisier-49045, ANGERS CEDEX 01}}

\affil[c]{\small{Department of Mathematical Sciences, Huzhou University, Zhejiang 313000, China}}

\maketitle

\begin{abstract}
\noindent
The necessary conditions for  an optimal control of a stochastic control problem with recursive utilities is investigated. The first order condition is the  the well-known Pontryagin type maximum principle. When the optimal control satisfying such first-order necessary condition is singular in some sense, certain type of the second-order necessary condition will come in naturally. The aim of this paper is to explore such kind of conditions for our optimal control problem.

\end{abstract}


\section{Introduction}

Consider a finite time horizon $T$. Let $(\Omega, {\mathscr F}, P)$ be a complete probability space and $W$  a $d$-dimensional standard
Brownian motion defined on this space. The filtration $\{{\mathscr F}_t\}_{0\leq t\leq T}$
is the natural filtraion generated by $W$ (augmented by all the $P$-null sets) that satisfies the usual condition. In this paper, we
consider the controlled system satisfying  the following stochastic differential equation (SDE for short) driven by
Brownian motion $\{W(t), 0\leq t\leq T\}.$:
\begin{equation}
\label{eq:1.1}
x(t)=x_0+\int_0^tb(s,x(s),u(s))ds+\int_0^t\sigma(s,x(s))dW_s.
\end{equation}
The associated cost functional is defined via the sulotion of a backward differential stochastic equation (BSDE for short):
\begin{equation}
\label{eq:1.2}
y(t)=h(x(T))+\int_t^Tf(s,x(s),y(s),
z(s),u(s))ds-\int_t^Tz(s)dW_s
\end{equation}
and given as
 \begin{equation}\label{eq:1.3}
   J(u(\cdot)):=y(0).
 \end{equation}
 In the context of mathematical finance, such functionals are sometimes called recursive utilities. We also call the solution $(y(\cdot),z(\cdot))$ of (\ref{eq:1.2}) the cost process associated with $u(\cdot)$. In the above system, $
b:[0,T]\times \mathbb R^n \times \bar U \rightarrow \mathbb R^n,
\sigma: [0,T] \times \mathbb R^n\rightarrow \mathbb R^{n\times d},
f: [0,T] \times \mathbb R^n \times \mathbb R \times \mathbb R^d \times \bar U \rightarrow \mathbb R,
h: \mathbb R^n \rightarrow \mathbb R$
are given fucntions with $U$ being  the control domain, that is assumed to be
a nonempty subset of ${\mathbb R}^m$ and not
necessarily to be convex, and  $\bar U$
its closure.  An admissible
control is defined as follows.
\begin{definition}
A control process $u (\cdot)$ is said to be admissible if it is an $U$-valued predictable process and satisfies
\begin{eqnarray*}
||u (\cdot)||_{{\cal U}_{\mathrm{ad}}} \triangleq \sup_{0 \leq t\leq T} \big\{\mathbb{E}
\big [ \left\vert u(t)\right\vert^{8} \big ]\big\}^{\frac{1}{8}}<\infty.
\end{eqnarray*}
Denote by ${\cal U}_{\mathrm{ad}}$ the set of all admissible control processes.
\end{definition}
The optimal control problem is to minimize the cost over ${{\cal U}_{\mathrm{ad}}}$, i.e.,
\begin{problem}\label{pro:2.1}
Find an admissible control $\bar u (\cdot)\in {\cal U}_{ad}$ such that
\begin{eqnarray}
J ( { \bar u} (\cdot) ) = \inf_{u (\cdot) \in {\cal U}_{ad}} J ( u (\cdot) )
\end{eqnarray}
subject to the state equation \eqref{eq:1.1}, \eqref{eq:1.2} and the cost functional (\ref{eq:1.3}).
\end{problem}

The process ${ \bar u} (\cdot)$ is called an optimal control.
The state  and cost processes  associated with ${\bar u} (\cdot)$, denoted by
$({\bar x}(\cdot), {\bar y}(\cdot),  {\bar z}(\cdot) )$, are called the optimal state  and cost processes.\\

 One tool  for the study of optimal control problems is the Pontryagin maximum principle which is to derive necessary conditions for the optimal pair. Before analyzing this issue in details, let us make some rough observations. Suppose $({\bar x}(\cdot), {\bar y}(\cdot),  {\bar z}(\cdot) )$ is an optimal pair of Problem \ref{pro:2.1}. For any given
 $u(\cdot)\in  {\cal
 	U}_{ad}$, let $u^\delta(\cdot)\in {\cal
 U}_{ad}$ be a suitable perturbation of $u(\cdot)$ determined by $u(\cdot)$ with a parameter $\delta$ (for examples, a convex type perturbation, or a spike type variation), so that $\rho(u^\delta(\cdot), \bar u(\cdot) )
 =O(\delta)$ with $\rho$ being a suitable metric on the set ${\cal U}_{ad}$, and the following holds:
 \begin{eqnarray}\label{eq:1.5}
   J(u^\rho(\cdot))=J(\bar u(\cdot))+\delta J_1(\bar u(\cdot),u(\cdot))+o(\delta).
 \end{eqnarray}
 Here $J_1(\bar u(\cdot),u(\cdot))$
 is some functional of $u(\cdot)$
 and $\bar u(\cdot).$
 The above can be called the first-order Taylor expansion of $J(\cdot)$ at $\bar u(\cdot)$, and $J_1(\bar u(\cdot),u(\cdot))$ can be regarded as the ''directional derivative'' of $J(\cdot)$ at $\bar u(\cdot)$ in the ''direction'' $u(\cdot)$. Hence, the minimality of $\bar u(\cdot)$ implies
 \begin{eqnarray}
   J_1(\bar u(\cdot),u(\cdot))\geq 0,
   ~~\forall u(\cdot)\in U.
 \end{eqnarray}
 Such a condition can be tranformed into the condition on the Hamiltonian (see  (\ref{Hamiltonian}) for the definition). It is called the first-order necessary condition for $\bar u(\cdot)$, which is essentially the Pontryagin's maximum principle. Sometimes, such a condition is sufficient to find the optimal control, for example, when there is only one control satisfies the condition. In other cases, the first order condition is insufficient especially when  the optimal control is singular. More precisely, suppose that there is a set
  ${\cal U}_0\subset {\cal U}_{\mathrm{ad}}$, which is different from the singleton, such that the following holds:
  \begin{eqnarray}
    J_1(\bar u(\cdot),u(\cdot))= 0, \forall u(\cdot)\in  {\cal U}_0.
  \end{eqnarray}
 Then $\bar u(\cdot)$ is said to be singular on the set ${\cal U}_0.$
 For convenience, we call ${\cal U}_0$ a singular set of $\bar u(\cdot)$. Let

 $${\cal U}_0(\bar u(\cdot))=\bigg \{u(\cdot)\in {\cal U}_{ad}|
    J_1(\bar u(\cdot),u(\cdot))= 0\bigg\},  $$
 which is called the maximum singular set of $\bar u(\cdot).$
 When ${\cal U}_0= {\cal U}_{\mathrm{ad}},$
we say that $\bar u(\cdot)$ is fully singular (or simply singular); When
 ${\cal U}_0(\bar u(\cdot))=\{\bar u(\cdot)\},$ we say that $\bar u(\cdot)$ is nonsingular; And, more interestingly, when  ${\cal U}_{\mathrm{ad}}\neq{\cal U}_0(\bar u(\cdot))\neq\{\bar u(\cdot)\},$  we say that
  $\bar u(\cdot)$
  is partially singular. The notion of singular control was introduced by Gabasov-Kirillova in \cite{GK}, where partial singularity was called ''the singularity in the sense of Pontryagin's maximum principle'', and full singularity was called ''the singularity in the classical sense''. We prefer to use the shorter names introduced by \cite{Lou}. Now, suppose
  $\bar u(\cdot)$ is partially singular. Then one should expect that the following (comparing with \eqref{eq:1.5})
 \begin{eqnarray}
   J(u^\rho(\cdot))=J(\bar u(\cdot))+\delta ^2J_2(\bar u(\cdot),u(\cdot))+o(\delta),  \forall u(\cdot)\in  {\cal U}_0,
 \end{eqnarray}
 for some functional $J_2(\bar u(\cdot),u(\cdot))$ of $(u(\cdot),\bar u(\cdot)).$ The above can be called the second-order Taylor expansion of $J(\cdot)$ at $u(\cdot)$ in the direction of $\bar u(\cdot)$, and $J_2(\bar u(\cdot),u(\cdot))$ can be regarded as the ¡°second order directional derivative¡± at $\bar u(\cdot)$ in the ¡°direction¡± $u(\cdot)$. Then the minimality of $\bar u(\cdot)$ leads to the following:

  \begin{eqnarray}
   J_2(\bar u(\cdot),u(\cdot))\geq 0,
   ~~\forall u(\cdot)\in {\cal U}_0,
    \end{eqnarray}
The purpose of this paper is to establish first and second order necessary optimality conditions for Problem
 \ref{pro:2.1} with recursive utilities. We shall calculate $J_2$ and transform the above condition into conditions on the Hamiltonian. It turns out to be a second order condition in some sense. \\

 Before we introduce the main results, let us first review the history on this
 topic.  When $f$ is
 independent of $(y,z),$
 it is easy to check that
 $y(0)=\mathbb E\bigg[h(x(T))+\int_t^Tf(s,x(s),u(s))ds\bigg]$ and then Problem
\ref{pro:2.1} becomes the
classical optimal control problem. We refer to
\cite{Ku} for an early study on the first-order necessary condition for stochastic optimal controls. After that, many authors contributed on this topic, see \cite{Be,Bis,Ha} and references cited therein. Compared to the deterministic setting, new phenomenon and difficulties appear when the diffusion term of the stochastic control system contains the control variable and the control region is nonconvex. The corresponding first-order necessary condition for this general case was established in \cite{Pe}.
For the recursive stochastic
optimal control problem, when the control
domain $U$ is convex, the local first-order maximum principle was studied
in \cite{Do99,Ji06,Pe93} (see also   \cite{Sh06,Wu98,Xu95}  and the references therein).
But for the general setting, it remained to be an open problem proposed  by Peng
\cite{Pe98} in a long time. By regarding $z(\cdot)$
as a control process and the
terminal condition $y(T)=h(x(T))$ as
a constraint and then using
the Ekeland variational principle,
Wu \cite{Wu13}
and Yong \cite{Yo10}
 established the corresponding first-order maximum
principles, but
contained unknown parameters in the formulation for the maximum principle.
 Recently, different from their methods, Hu \cite{Hu17}  completely solved this problem by establishing the variation equation for backward stochastic differential equations.\\

 As we see in the previous, for the singular control, it may happen that the first-order necessary conditions turn out to be trivial. Either the gradient and the Hessian of the corresponding Hamiltonian with respect to the control variable vanish/degenerate or  the Hamiltonian is equal to a constant in the control region. In these cases, the first-order necessary condition cannot provide enough information for the theoretical analysis and numerical computing, and therefore one needs to study the second-order necessary conditions.
Along the line of necessary conditions for singular optimal control problems, the deterministic case was considered by many authors. The reader is referred to Bell and Jacobson \cite{Be75}, the review paper by Gabasov and Kirillova \cite{GK} (and the references therein) for relevant results, Kazemi-Dehkordi \cite{Ka84}
and  Krener \cite{Kr77}. Compared to the deterministic control systems,  second-order necessary condition for stochastic optimal controls was first investigated by Tang \cite{Ta10}. In
\cite{Ta10}, a pointwise second-order maximum principle for stochastic singular optimal controls in the sense of Pontryagin-type maximum principle was established which involves second-order adjoint processes, for the case that the diffusion term $ \sigma(t, x, u)$ is independent of the control $u$, via a generalized spike variation technique together with the vector-valued measure theory and the second-order expansions of both the system and the cost functional. Recently, this direction has drawn great attention, see \cite{Bo12,Fr17,Zh15,Zh17}. In \cite{Bo12}, an integral-type second-order necessary condition for stochastic optimal controls was derived under the assumption that the control region $U$ is convex.
While in \cite{Zh15}, a pointwise second-order necessary condition for stochastic optimal controls is established in the case that  both drift and diffusion terms may contain the control variable $u$, and
the control region $U$ is still assumed to be convex.  The method was further developed in \cite{Zh15} to obtain a pointwise second-order necessary condition in general cases where the control region is allowed to be nonconvex, but the analysis there is much more complicated, see also \cite{Fr17} and \cite{Zh17} for details.\\

This paper is first to investigate
the second-order maximum principle for
the recursive optimal control problem.
We established a pointwise second-order condition in the sense of Pontryagin-type maximum principle with a nonconvex control region when the diffusion term is independent of the control $u$. Via a generalized spike variation technique together with the vector-valued measure theory,  we gave the second-order expansions of both the system and the cost functional and the second-order dual process  which are of interest themselves. Finally, the analysis leads to the main results that contains the result of \cite{Ta10}. The rest of this paper is organized as follows. In Section 2, we introduce the formulation of the optimal control problem and  give the main results of this paper. Section 3 includes a quantitative analysis for the variations of the system and the cost between two different control actions. Section 4 contains the proof for the necassary condtions both of the first and second order. Section 5 provides some examples.

\section{Formulation of the Problem and the Main Results}
\subsection{Notations}
We consider a finite time horizon $T$ and a complete probability space $(\Omega, {\cal F}, {\mathbb P})$
carrying a $d$-dimensional standard Brownian motion $W (\cdot) := \{ W (t) | t \in [0, T] \}$. Without loss of generality, we assume that $d=1$ for simplicity of the presentation. Let ${\mathbb F} : = \{ {\cal F}_t | t \in [0, T] \}$ be a filtration generated by $W (\cdot)$
and satisfying the usual conditions of right-continuity and ${\mathbb P}$-completeness. We denote by $\mathcal {P}$
the predictable $\sigma$-field on $[0, T] \times \Omega$, and $\mathcal {B} (\Lambda)$
the Borel $\sigma$-algebra of any topological space $\Lambda$. Let $\mathbb H$ be an Euclidean space,
in which the inner product and the norm is denoted by $\left < \cdot, \cdot
\right >$ and $| \cdot |$, respectively. We denote the points in $\mathbb H$ as a column vector. Given a matrix $A \in \mathbb R^{n \times n}$ and $x \in \mathbb R^n$, we denote by $A(x)^2:=\left <Ax,x\right>$. For a function $\phi: {\mathbb R}^n \rightarrow {\mathbb R}$,
we use $\phi_x$ to denote its gradient and $\phi_{xx}$ its Hessian (a symmetric matrix). If
$\phi: {\mathbb R}^n \rightarrow {\mathbb R}^k$, where $k \geq 2$,
then $\phi_x = [ \frac{\partial \phi_i}{\partial x_j} ]_{i = 1, 2, \cdots, k; j = 1, 2, \cdots, n}$
is the corresponding $(k \times n)$-Jacobian matrix. Furthermore, we denote by  $A^*$  the transpose
of any vector or matrix $A$, and $C$ and $K$ two generic positive constants, which may be different from line to line.\\

Several spaces of random variables and stochastic processes on $( \Omega, {\cal F}, \mathbb {P} )$ will be
used throughout the paper. For any $\alpha, \beta \in [1,\infty)$, we define

\begin{itemize}

\item $L_{\mathbb {F}}^\beta ( 0, T; \mathbb {H})$: the space of all
$\mathbb {H}$-valued and $\mathbb {F}$-adapted processes $f (\cdot) = \{ f ( t, \omega ) | ( t, \omega )
\in [ 0, T ] \times \Omega \}$ such that $\| f (\cdot) \|_{ L_{\mathbb {F}}^\beta ( 0, T; {\mathbb H} )} \triangleq
\left \{ {\mathbb E} \left [ \int_0^T | f (t) |^\beta d t \right ] \right \}^{\frac{1}{\beta}} < \infty$;

\item $S_{\mathbb {F}}^\beta ( 0, T; \mathbb {H} )$: the space of all $\mathbb {H}$-valued, $\mathbb {F}$-adapted,
c\`adl\`ag processes $f (\cdot) = \{ f ( t, \omega ) | ( t, \omega ) \in [ 0, T ] \times \Omega \}$ such that
$\| f (\cdot) \|_{ S_{\mathbb {F}}^\beta ( 0, T; {\mathbb H} )} \triangleq
\left \{ {\mathbb E} \left [ \sup_{ t \in [0, T] } | f (t) |^\beta \right ] \right \}^{\frac{1}{\beta}} < \infty$;

\item $L^\beta_{{\cal F}_T} ( \Omega; \mathbb {H})$: the space of all $\mathbb {H}$-valued, ${\cal F}_T$-measurable
random variables $\xi$ on $( \Omega, {\cal F}, \mathbb {P})$ such that $\| \xi \|_{L^\beta_{{\cal F}_T} ( \Omega; \mathbb {H})} \triangleq \left \{ {\mathbb E} \left [ | \xi |^\beta \right ] \right \}^{\frac{1}{\beta}} < \infty$;

\item $L_{\mathbb {F}}^\beta ( 0, T; L^\alpha ( 0, T; {\mathbb H}))$: the space of all
$L^\alpha ( 0, T; {\mathbb H})$-valued, ${\mathbb {F}}$-adapted processes $f (\cdot) =
\{ f ( t, \omega ) | (t,\omega) \in [ 0, T ] \times \Omega \}$ such that
$\| f (\cdot) \|_{ L_{\mathbb {F}}^\beta ( 0, T; L^\alpha ( 0, T; {\mathbb H}))}
\triangleq \left \{ {\mathbb E} \left [ \left ( \int_0^T | f (t) |^\alpha d t \right )^{\frac{\beta}{\alpha}} \right ]
\right \}^{\frac{1}{\beta}} < \infty$.

\end{itemize}
In addition, we write $M^p_{\mathbb {F}} [ 0, T ] \triangleq S_{\mathbb {F}}^p ( 0, T; \mathbb {R}^n )
\times S_{\mathbb {F}}^p ( 0, T; \mathbb {R} ) \times L_{\mathbb {F}}^p ( 0, T; L^2 (0, T; \mathbb {R}^{ d}) )$.
Clearly, $M^p_{\mathbb F}[0,T]$ is a Banach space. For any triplet of processes $\Theta (\cdot) \triangleq ( x (\cdot), y (\cdot), z (\cdot) )$
in $M^p_{\mathbb {F}} [0,T]$, the corresponding norm is defined as
\begin{eqnarray*}
\| \Theta (\cdot) \|_{ M^p_{\mathbb {F}} [ 0, T ] }
\triangleq \left \{ {\mathbb E} \left[ \sup_{ t \in [ 0, T ] } |x(t)|^p
+ \sup_{ t \in [ 0, T ] } | y (t) |^p + \left ( \int_0^T | z (t) |^2 d t \right)^{\frac{p}{2}} \right ] \right \}^{\frac{1}{p}} .
\end{eqnarray*}

\subsection{Basic Assumptions}
 In this subsection, we introduce
 some basic assumptions on the coefficients of our control problem.
Let $K_0$ be some positive constant.
\begin{assumption}\label{ass:1.1}
	The functions $b,\sigma, h, f$ are Borel measurable with respect to their respective arguments, continuous in $u$,
continuously differentiable in $(x,y,z)$ for each fixed $(t, u)$, and
 \begin{equation}
	\label{assumption1 SDE}
	\begin{split}	&|b_x(t,x,u)|,|\sigma_x(t,x)|,|h_x(x)|,
|f_x(t,x,y,z,u)|,|f_y(t,x,y,z,u)|,
|f_z(t,x,y,z,u)| \le K_0,\\
	&|b(t,x,u)| \le K_0(1+|x|+|u|),|\sigma(x,u)|\le K_0(1+|x|+|u|),|h(x)| \le K_0(1+|x|),\\&|f(t,x,y,z,u)|\le K_0(1+|x|+|y|+|z|+|u|).
	\end{split}
	\end{equation}
Moreover, all the derivatives involved above are Borel measurable, and are continuous in $x$.

\end{assumption}
\begin{assumption}\label{ass:1.2}
 The first-order derivatives involved above are continuous in $u$ on $\bar U$. The functions $ b, \sigma, f$ and $h$ have continuous second-order derivatives in $x$. The second-order derivatives are Borel measurable with respect to $(t, x, y,z, u)$, and are bounded by the constant $K_0$, that is
\begin{equation}
\label{assumption2 BSDE}
\begin{split}
|b_{xx}(t,x,u)|
,|\sigma_{xx}(t,x)|,|f_{xx}(t,x,y,z,u)|,
|f_{xx}(t,x,y,z,u)|,|f_{xy}(t,x,y,z,u)|,
|h_{yy}(x)|\le K_0.
\end{split}
\end{equation}
\end{assumption}

 For each $u(\cdot) \in U_{ad}$, the SDE \eqref{eq:1.1} and BSDE \eqref{eq:1.2}, under the Assumption \ref{ass:1.1}, have a unique strong solution, which will be denoted by $(x(\cdot;u(\cdot)),y(\cdot;u(\cdot)), z(\cdot;u(\cdot)))
 \in M^8_{\mathbb {F}} [ 0, T ] \triangleq S_{\mathbb {F}}^8 ( 0, T; \mathbb {R}^n )
\times S_{\mathbb {F}}^8 ( 0, T; \mathbb {R} ) \times L_{\mathbb {F}}^8 ( 0, T; L^2 (0, T; \mathbb {R}^{d}) ),
$ or simply $(x(\cdot),y(\cdot),z(\cdot))$ if its dependence on the admissible control $u(\cdot)$ is clear from the context.\\

For future purposes, we recall the standard estmates of BSDEs (see \cite{Hu17} and the refereneces therein).
\begin{lemma}
	\label{lemma standard esimate BSDE}
	Let $(Y_i,Z_i),i=1,2$, be the solutions of the following BSDEs:
	\[Y_i(t)=\xi_i+\int_t^Tf_i(s,Y_i(s),Z_i(s))ds-\int_t^TZ_i(s)dW_s,\]
	where $E[|\xi_i|^{\beta}]<\infty,f_i=f_i(s,\omega,y,z):[0,T] \times \Omega \times \mathbb R\times \mathbb R^d \rightarrow \mathbb R$ is progressively measurable for each fixed $(y,z)$, Lipschitz in $(y,z)$, and $E[(\int_0^T|f_i(s,0,0)|ds)^\beta]<\infty$ for some $\beta >1$. Then there exists a constant $C_{\beta}$ depending on $\beta, T$ and the Lipschitz constant such that
	\begin{equation*}
	\begin{split}
	&\mathbb E\bigg[\sup_{t \in [0,T]}|Y_1(t)-Y_2(t)|^\beta
+\bigg(\int_0^T|Z_1(s)-Z_2(s)|^2ds
\bigg)^{\beta/2}\bigg]\\
	\le &C_{\beta}\mathbb E\bigg[|\xi_1-\xi_2|^{\beta}
+\bigg(\int_0^T|f_1(s,Y_1(s),Z_1(s))
-f_2(s,Y_1(s),Z_1(s))|ds\bigg)^\beta\bigg].
	\end{split}
	\end{equation*}
In particular, taking $\xi_1=0$ and $f_1=0$, we have
\[\mathbb E\bigg[\sup_{t \in [0,T]}|Y_2(t)|^\beta+(\int_0^T|Z_2(s)|^2ds)^{\beta/2}]\le
C_{\beta}\mathbb E\bigg[|\xi_2|^{\beta}
+(\int_0^T|f_2(s,0,0)|ds)^\beta\bigg].\]
\end{lemma}

\subsection{the Main Results}
The object of this paper is to establish a general maximum principle for
Problem \ref{pro:2.1}. When the convexity assumption is not made on the control domain $U$,
the basic idea of deriving necessary conditions is to apply the spike variation
to the control process and derive a Taylor-type expansion for the state
process and the cost functional with respect to the spike variation of the control process.
Then using some suitable duality relations, one can obtain  a maximum principle of Pontryagin's type.\\

Define the Hamiltonian:
\begin{equation}
\label{Hamiltonian}
H(t,x,y,z,u,p,q):=\langle p,b(t,x,u)\rangle +\langle q,\sigma(t,x)\rangle +f(t,x,y,z,u).
\end{equation}
Let $\bar u(\cdot)$ be an optimal control and $(\bar x(\cdot),\bar y(\cdot),\bar z(\cdot))$ the associated state and cost process. To simplify the notations, we introduce the following abbreviations:
 \begin{eqnarray}
   \begin{split}
     \bar b(t):=b(t,\bar x(t),\bar u(t)), b(t;u):=b(t,\bar x(t),u), \delta b(t;u):=b(t,\bar x(t),u)-\bar b(t)
   \end{split}
 \end{eqnarray}
and define similarly for $\bar b_x(t),\bar b_{xx}(t),\delta b_x(t;u),\bar \sigma(t),\bar \sigma_{x}(t),\bar \sigma_{xx}(t),\bar f (t),\bar f_x(t)$, $\bar f_y(t)\bar f_z(t),\delta f(t;u)$ and so on. We introduce respectively the following two adjoint equations:
\begin{equation}
\label{eq:3.1}
\left \{
\begin{split}
&dp(t)=-\big\{\big[\bar f_y(t)+\bar f_z(t)\sigma^*_x(t)+\bar b^*_x(t)\big]p(t)+\big[\bar f_z(t)+\sigma^*_x(t)\big]q(t)+\bar f^*_x(t)\big\}dt+q(t)dW_t,\\
&p(T)=h^*_x(\bar x_T),
\end{split}
\right .
\end{equation}
and
\begin{equation}
\label{eq:3.2}
\left \{
\begin{split}
&dP_t=-\big\{\bar f_y(t)P(t)+\big[\bar f_z(t)\bar\sigma_x(t)+\bar b_x(t)\big]^*P(t)+P(t)\big[\bar f_z(t)\bar\sigma_x(t)+\bar b_x(t)\big]+\bar {\sigma}^*_x(t)P(t)\bar {\sigma}_x(t)\\
&\qquad \quad+\bar f_z(t)Q(t)+\sigma^*_x(t)Q(t)+Q(t)\sigma_x(t)+p^*(t)\bar b_{xx}(t)+\big[\bar f_z(t)p(t)+q(t)\big]^*\bar \sigma_{xx}(t)\\
&\qquad \quad+[I,p(t),\bar \sigma^*_x(t)p(t)+q(t)]D^2 \bar f(t)[I,p(t),\bar \sigma^*_x(t)p(t)+q(t)]^T\big\}dt+Q(t)dW_t,\\
&P(T)=h_{xx}(\bar x_T),
\end{split}
\right .
\end{equation}
where $D^2f$ is the Hessian matrix of $f$ with respect to $(x,y,z)$. \\

Under Assumptions \ref{ass:1.1}
and \ref{ass:1.2}, from Lemma \ref{lemma standard esimate BSDE}, it is easy to see that for any  admissible pair
$(\bar u(\cdot), \bar x(\cdot), \bar y(\cdot), \bar z(\cdot))$, BSDEs \eqref{eq:3.1} and \eqref{eq:3.2}
admit unique solutions $(p (\cdot), q (\cdot)) \in S^8_{\mathbb F} (0, T; \mathbb {R}^n) \times
L^8_{\mathbb F} (0, T;L^2 (0, T; \mathbb {R}^{n \times d}))$ and $(P (\cdot), Q (\cdot)) \in
S^8_{\mathbb F} (0, T; \mathbb {R}^{n \times n}) \times L^8_{\mathbb F} (0, T; L^2 (0, T;
(\mathbb {R}^{n \times n})^d))$, respectively. We call \eqref{eq:3.1} and \eqref{eq:3.2}
the first-order and the second-order adjoint equations
of the control system \eqref{eq:1.1}-\eqref{eq:1.2}, respectively, where the unique adapted solutions
$(p (\cdot), q (\cdot))$ and $(P (\cdot), Q (\cdot))$ are referred as the first-order and the second-order
adjoint processes.
We also use the abbreviations:
\begin{eqnarray}
  \begin{split}
    &H(t)=H(t,\bar x(t),\bar y(t),\bar z(t),\bar u(t), p(t),q(t)),
  \\&H_x(t)=H_x(t,\bar x(t),\bar y(t),\bar z(t),\bar u(t), p(t),q(t)), \\&H_{xx}(t)=H_{xx}(t,\bar x(t),\bar y(t),\bar z(t),\bar u(t), p(t),q(t)),
  \\& \delta H(t,v)=H(t,\bar x(t),\bar y(t),\bar z(t), v, p(t),q(t))-H(t,\bar x(t),\bar y(t),\bar z(t), \bar u(t), p(t),q(t)).
  \end{split}
\end{eqnarray}

In the following, we state the main
results of our paper.  The first
is the   first-order maximum principle.

\begin{theorem} \label{thm:1.1}
Let Assumption \ref{ass:1.1} be satisfied. Let
  $(\bar u(\cdot); \bar x(\cdot),\bar y(\cdot),\bar z(\cdot))$ be an optimal pair. Then there is a subset $I_0 \subset [0, 1]$ which is of full measure, such that at each $t\in I_0 $ the minimum condition

  \begin{eqnarray}\label{eq:3.5}
  H(t,\bar x(t),\bar y(t),\bar z(t), \bar u(t), p(t),q(t))=\min_{v\in U} H(t,\bar x(t),\bar y(t),\bar z(t), v, p(t),q(t)),
  a.s.
  \end{eqnarray}
  holds.
\end{theorem}

The maximum principle is a powerful tool for the study of optimal stochastic control problems. However, it is not always effective. For example, if the optimal admissible pair
$(\bar u(\cdot); \bar x(\cdot),\bar y(\cdot),\bar z(\cdot))$
is such that $h_x(\bar x(T)) = 0, f_x(t, \bar x(\cdot),\bar y(\cdot),\bar z(\cdot), \bar u(\cdot)) = 0, $ a.e. a.s..
In this case, the adjoint process
$(p(\cdot), q(\cdot))$, defined by BSDE
\eqref{eq:3.1}, is identically zero, and the maximum condition \eqref{eq:3.5} is trivial, giving no information about the optimal control $u(\cdot)$. Such a control $u(\cdot)$ is a singular one. There are other kinds of singular controls, for which the above maximum principle is ineffective.
In this paper, we discuss singular optimal stochastic controls in the following sense of maximum principle.

\begin{definition}
An admissible control $\tilde u$ is called singular on control region $V$ if $V \subset U$ is nonempty and for $\text{a.e. } t \in [0,T]$, we have
\[H(t,\tilde x(t),\tilde y(t),\tilde z(t),\tilde u(t),\tilde p(t),\tilde q(t))=H(t,\tilde x(t),\tilde y(t),\tilde z(t),v,\tilde p(t),\tilde q(t)), \forall v \in V.\]
\end{definition}

The main result of this paper is the following second-order maximum principle which  involves the second-order adjoint processes $(P(\cdot), Q(\cdot))$  given in (\ref{eq:3.2}).
\begin{theorem}\label{thm:2.2}
Let Assumption \ref{ass:1.1} and
 \ref{ass:1.2} be satisfied. Let
  $(\bar u(\cdot); \bar x(\cdot),\bar y(\cdot),\bar z(\cdot))$ be an optimal pair and be singular on the control region $V$. Then there exists $I_0 \subset [0,1]$ which is of full measure, such that at each $t \in I_0$, $\bar u(\cdot)$ satisfies, in addition to the first order maximum condtion, the following second order maximum condition:
\begin{equation}\label{second order condition}
(\delta G(t;v)+\delta b^*(t;v)P(t))\delta b(t;v) \ge 0, \forall v  \in V, a.s.,
\end{equation}
where we have used the following short-hand notation:
\[G(t;u):=H_x(t;u)+\bar f_y(t;u)p^*(t)+\bar f_z(t;u)(p^*(t)\bar \sigma_x(t)+q^*(t)).\]
\end{theorem}

\section{First and Second Order Taylor Expansion}
In this section, we introduce the
first and the
second order variation equation for the optimal
pair $(\bar u(\cdot); \bar x(t),\bar y(t),\bar z(t))$
by spike variation methods and
establish  the dependence of the system state  on control actions.\\

Let $u(\cdot)\in {\cal U}_{ad}$, $\varepsilon >0$ and $E_{\varepsilon} \subset [0,T]$ be a Borel set with Borel measure $|E_{\varepsilon}|=\varepsilon$. Define the spike variation $u^{\varepsilon}$ of the optimal control $\bar u$ as
\[u^\varepsilon(t)=\bar u(t)I_{E^c_{\varepsilon}}(t)+u(t)
I_{E_{\varepsilon}}(t).\]

Let $x_i(\cdot),i=1,2$, be the solution for the following SDEs which is regarded as the corresponding first  and
second order variation equations for the optimal
state process $\bar x(\cdot)$:
:
\begin{equation}
\label{first oder SDE}
\left \{
\begin{split}
&dx_1(t)=\{\bar b_x(t)x_1(t)+\delta b(t;u^{\varepsilon}(t))\}dt+\bar \sigma_x(t)x_1(t)dW_t,\\
&x_1(0)=0
\end{split}
\right.
\end{equation}
and
\begin{equation}
\label{second order SDE}
\left \{
\begin{split}
&dx_2(t)=\{\bar b_x(t)x_2(t)+\delta b_x(t;u^{\varepsilon}(t))x_1(t)+\frac{1}{2}\bar b_{xx}(t)(x_1(t))^2\}dt\\
&\qquad \qquad +\{\bar \sigma_x(t)x_2(t)+\frac{1}{2}\bar \sigma_{xx}(t)(x_1(t))^2\}dW_t,\\
&x_2(0)=0,
\end{split}
\right .
\end{equation}
where $b_{xx}(t)(x_1(t))^2=(b^1_{xx}(t)(x_1(t))^2,...,b^n_{xx}(t)(x_1(t))^2)^T$ and similarly for $\sigma_{xx}(t)(x_1(t))^2$.\\

The following lemma is a standard result and has been proved in \cite{Ta10}.
\begin{lemma}
	\label{lemma estiamtion SDE}
	Assume that Assumption 1 and Assumption 2 are satisfied.  Then we have
	\begin{equation*}
	\begin{split}
	&\mathbb E\bigg[\sup_{0\le t \le T}|x(t;u^{\varepsilon})-\bar x(t)|^8\bigg]=O(\varepsilon^8),\\	
	&\mathbb E\bigg[\sup_{0\le t \le T}|x_1(t)|^8\bigg]=O(\varepsilon^8),\\
		&\mathbb E\bigg[\sup_{0\le t \le T}|x(t;u^{\varepsilon})-\bar x(t)-x_1(t)|^2\bigg]=O(\varepsilon^{4}),\\
		&\mathbb E\bigg[\sup_{0\le t \le T}|x_2(t)|^2\bigg]=O(\varepsilon^{4}),\\
		&\mathbb E\bigg[\sup_{0\le t \le T}|x(t;u^{\varepsilon})-\bar x(t)-x_1(t)-x_2(t)|^2\bigg]=o(\varepsilon^{4}).\\
		\end{split}
	\end{equation*}
\end{lemma}
Let $(y_1,z_1)$ be the solution of the following BSDE:
\begin{equation}
\label{BSDE First Order}
\left \{
\begin{split}
&dy_1(t)=-\{\bar f_y(t)y_1(t)+\bar f_z(t)z_1(t)+p^*(t)\delta b(t;u^{\varepsilon}(t))+\delta f(t;u^{\varepsilon}(t))\}dt+z_1(t)dW_t\\
&y_1(T)=0.
\end{split}
\right .
\end{equation}

\begin{lemma}
	\label{lemma first order}
Assume Assumption \ref{ass:1.1} to be satisfied. Then
the following estimation holds:
\[\mathbb E\bigg[\sup_{0 \le t \le T}|y^{\varepsilon}(t)-\bar y(t)-p^*(t)x_1(t)-y_1(t)|^4\bigg]=o(\varepsilon^4).\]
\end{lemma}
\begin{proof}
 Define
\[\tilde y^{\varepsilon}(t):=y^{\varepsilon}(t)-\bar y(t)-p^*(t)x_1(t)-y_1(t)\]
and
\[\tilde z^{\varepsilon}(t):=z^{\varepsilon}(t)-\bar z(t)-p^*(t)\bar \sigma_x(t)x_1(t)-q^*(t)x_1(t)-z_1(t).\]
Applying It\^o's formular to $\tilde y^{\varepsilon}$, we have
\[d\tilde y^{\varepsilon}(t)=-I(t)dt+\tilde z^{\varepsilon}(t)dW_t\]
with
\begin{equation*}
\begin{split}
I(t):=&f(t,x^{\varepsilon}(t),y^{\varepsilon}(t),z^{\varepsilon}(t),u^{\varepsilon}(t))-\bar f(t)-\bar f_x(t)x_1(t)-\bar f_y(t)(p^*(t)x_1(t)+y_1(t))\\
&-\bar f_z(t)\big[p^*(t)\bar \sigma_x(t)x_1(t)+q^*(t)x_1(t)+z_1(t)\big]-\delta f(t;u^{\varepsilon}_t).
\end{split}
\end{equation*}
Thus, we see that
\begin{equation*}
\begin{split}
&f(t,x^{\varepsilon}(t),y^{\varepsilon}(t),z^{\varepsilon}(t),u^{\varepsilon}(t))-\bar f(t)-\delta f(t;u^{\varepsilon}(t))\\
=&f(t,x^{\varepsilon}(t),y^{\varepsilon}(t),z^{\varepsilon}(t),u^{\varepsilon}(t))-f(t,\bar x(t),\bar y(t),\bar z(t),u^{\varepsilon}(t))\\
=&\bar f_x(t)(x^{\varepsilon}(t)-\bar x(t))+\bar f_y(t)(y^{\varepsilon}(t)-\bar y(t))+\bar f_z(t)(z^{\varepsilon}(t)-\bar z(t))+i(t),
\end{split}
\end{equation*}
where the residual term $i(t)$ satisfies
\[\mathbb E\bigg[\bigg(\int_0^T |i(t)|dt\bigg)^4\bigg]=o(\varepsilon^4).\]
Hence
\[I(t)=\bar f_x(t)(x^{\varepsilon}(t)-\bar x(t)-x_1(t))+\bar f_y(t)\tilde y^{\varepsilon}(t)+\bar f_z(t) \tilde z^{\varepsilon}(t)+i(t).\]
The starndard estimate of BSDEs yields that
\[\mathbb E\bigg[\sup_{0 \le t \le T}|\tilde y^{\varepsilon}(t)|^4\bigg]\le C\mathbb E\bigg[\bigg(\int_0^T|x^{\varepsilon}(t)-\bar x(t)-x_1(t)|+|i(t)|dt\bigg)^4\bigg]=o(\varepsilon^4).\]
\end{proof}
 To derive the second order condition in the next section, we also need to  expand the value function to the second order. Let $(y_2,z_2)$ be the solution of the following:
\begin{equation}
\label{BSDE second order}
\left \{
\begin{split}
&dy_2(t)=-\bigg\{\bar f_y(t)y_2(t)+\bar f_z(t)z_2(t)+\left<P(t)\delta b(t;u^{\varepsilon}_t),x_1(t)\right>+p^*(t)\delta b_x(t;u^{\varepsilon}_t)x_1(t)\\
&\qquad \qquad  +\bigg[\delta f_x(t;u^{\varepsilon}(t))+\delta f_y(t;u^{\varepsilon}(t))p^*(t)+\delta f_z(t;u^{\varepsilon}(t))(p^*(t)\bar \sigma_x(t)+q^*(t))\bigg]x_1(t)\bigg\}dt+z_2(t)dW_t,\\
&y_2(T)=0.
\end{split}
\right .
\end{equation}
 We now establish the following lemma.

\begin{lemma}
Assume that Assumption 1 and Assumption 2 are satisfied. Let $\bar u(\cdot)$ be a optimal control singular on the control region $V$ and $u(\cdot)$ any $V$-valued admissible control. For any $r>1$, we have
\begin{eqnarray}
\begin{split}
  \mathbb E\bigg[\sup_{0 \le t \le T}|y^{\varepsilon}(t)-\bar y(t)-p^*(t)(x_1(t)+x_2(t))
-\frac{1}{2}P(t)(x_1(t))^2
-y_2(t)|^2\bigg]=o(\varepsilon^4).
\end{split}
\end{eqnarray}
\end{lemma}

\begin{proof}
 Note that for any $V$-valued admissible control $u$, the corresponding process $y_1$ satisfies $y_1(t) \equiv 0$. Hence, from Lemma \ref{lemma first order}, we have
\begin{eqnarray}
\mathbb E\bigg[\sup_{0 \le t \le T}|y^{\varepsilon}(t)-\bar y(t)-p^*(t)x_1(t)|^4\bigg]=o(\varepsilon^4)
\end{eqnarray}
and

\begin{eqnarray}
\mathbb E\bigg[\bigg(\int_0^T|z^\varepsilon(t)-\bar z(t)-(p^*(t)\bar\sigma_x(t)+q^*(t))x_1(t)|^2dt\bigg)^2
\bigg]=o(\varepsilon^4).
\end{eqnarray}

Applying It\^o's formula, we have
\begin{equation}
\label{Ito px_1x_2}
\begin{split}
&d\bigg[p^*(t)(x_1(t)+x_2(t))\bigg]\\
=&\bigg\{p^*(t)\delta b(t;u^{\varepsilon}(t))+p^*(t)\delta b_x(t)x_1(t)-\bar f_y(t)\bigg[p^*(t)(x_1(t)+x_2(t))\bigg]\\
&-\bar f_z(t)\bigg[p^*(t)\bar \sigma_x(t)(x_1(t)+x_2(t))+q^*(t)\bigg]-\bar f_x(t)\bigg[x_1(t)+x_2(t)\bigg]\\
&+\frac{1}{2}\bigg[p^*(t) \bar b_{xx}(t)+q^*(t)\bar\sigma_{xx}(t)\bigg](x_1(t))^2\bigg\}dt+\bigg\{q^*(t)\bigg[x_1(t)+x_2(t)\bigg]\\
&+p^*(t)\bigg[\bar \sigma_x(t)(x_1(t)+x_2(t))+\frac{1}{2}\bar \sigma_{xx}(t)x_1(t)\otimes x_1(t)\bigg]\bigg\}dW_t
\end{split}
\end{equation}
and
\begin{equation}
\label{Ito Px_1x_1}
\begin{split}
&d\bigg[\frac{1}{2}P(t)(x_1(t))^2\bigg]\\
=&\bigg\{\left <P(t)\delta b(t;u^{\varepsilon}(t)),x_1(t)\right>-\frac{1}{2}\bar f_y(t)P(t)(x_1(t))^2-\frac{1}{2}\bar f_z(t)\bigg[\bar \sigma^*_x(t)P(t)+P(t)\sigma_x(t)+Q(t)\bigg](x_1(t))^2\\
& -\frac{1}{2}\bigg[p^*(t)\bar b_{xx}(t)+\big[\bar f_z(t)p(t)+q(t)\big]^*\bar \sigma_{xx}(t)+[I,p(t),\bar \sigma_x(t)p(t)\\
& +q(t)]D^2 \bar f(t)[I,p(t),\bar \sigma_x(t)p(t)+q(t)]^T\bigg]
(x_1(t))^2\bigg\}dt
+\bigg\{\frac{1}{2}\bar f_z(t)\bigg[\bar \sigma^*_x(t)P(t)+P(t)\sigma_x(t)+Q(t)\bigg](x_1(t))^2\bigg\}dW_t.
\end{split}
\end{equation}
Define
\[\hat y^{\varepsilon}(t):=y^{\varepsilon}(t)-\bar y(t)-p(t)(x_1(t)+x_2(t))-\frac{1}{2}P(t)
(x_1(t))^2-y_2(t)\]
and
\begin{equation*}
\begin{split}
\hat z^{\varepsilon}(t):=&z^{\varepsilon}(t)-\bar z(t)-\bigg\{q^*(t)\bigg[x_1(t)+x_2(t)\bigg]
+p^*(t)\bigg[\bar \sigma_x(t)(x_1(t)+x_2(t))+\frac{1}{2}\bar \sigma_{xx}(t)x_1(t)\otimes x_1(t)\bigg]\bigg\}\\
&\qquad \qquad-\frac{1}{2}\bar f_z(t)\bigg[\bar \sigma^*_x(t)P(t)+P(t)\sigma_x(t)+Q(t)\bigg](x_1(t))^2.
\end{split}
\end{equation*}
Moreover, using Taylor expansion of $f$, we have
\begin{equation}
\label{Taylor expand f second order 1}
\begin{split}
&f(t,x^{\varepsilon}(t),y^{\varepsilon}(t),z^{\varepsilon}(t),u^{\varepsilon}(t))-f(t,\bar x(t),\bar y(t),\bar z(t),u^{\varepsilon}(t))\\
=&f_x(t;u^{\varepsilon}(t))(x^{\varepsilon}(t)-\bar x(t))+f_y(t;u^{\varepsilon}(t))(y^{\varepsilon}(t)-\bar y(t))+f_z(t;u^{\varepsilon}(t))(z^{\varepsilon}(t)-\bar z(t))\\
&+\frac{1}{2}[x^{\varepsilon}(t)-\bar x(t),y^{\varepsilon}(t)-\bar y(t),z^{\varepsilon}(t)-\bar z(t)]D^2f(t;u^{\varepsilon}(t))
[x^{\varepsilon}(t)-\bar x(t),y^{\varepsilon}(t)-\bar y(t),z^{\varepsilon}(t)-\bar z(t)]^T\\
&+i_1(t),
\end{split}
\end{equation}
where $i_1(t)$ is the residual term of Taylor expansion, one can easily obtain that
\[\mathbb E\bigg[\bigg(\int_0^T|i_1(t)|dt\bigg)^2\bigg]=o(\varepsilon^4)\].
Also, we see that
\[f_x(t;u^{\varepsilon}(t))(x^{\varepsilon}(t)-\bar x(t))=\bar f_x(t)(x^{\varepsilon}(t)-\bar x(t))+\delta f_x(t;u^{\varepsilon}(t))x_1(t)+i_2(t),\]
where $i_2(t)$ also satisfies
\[\mathbb E\bigg[\bigg(\int_0^T|i_2(t)|dt\bigg)^2\bigg]=o(\varepsilon^4).\]
We can get similar approaximations for the terms of $y$ and $z$ and the quadratic term. Thus, finally we rewrite (\ref{Taylor expand f second order 1}) as
\begin{equation}
\label{Taylor expand f second order 2}
\begin{split}
&f(t,x^{\varepsilon}(t),y^{\varepsilon}(t),z^{\varepsilon}(t),u^{\varepsilon}(t))-f(t,\bar x(t),\bar y(t),\bar z(t),u^{\varepsilon}(t))\\
=&\bar f_x(t)(x^{\varepsilon}(t)-\bar x(t))+\bar f_y(t)(y^{\varepsilon}(t)-\bar y(t))+\bar f_z(t)(z^{\varepsilon}(t)-\bar z(t))+\delta f_x(t;u^{\varepsilon}(t))x_1(t)\\
&+\delta f_y(t;u^{\varepsilon}(t))p^*(t)x_1(t)+\delta f_z(t;u^{\varepsilon}(t))(p^*(t)\bar\sigma_x(t)+q^*(t))x_1(t)\\
&+\frac{1}{2}[I,p^*(t),p(t)\bar\sigma_x(t)+q(t)]D^2\bar f(t)[I,p(t),p(t)\bar\sigma_x(t)+q(t)]^T(x^1(t))^2\\
&+i_3(t),
\end{split}
\end{equation}
with $i_3(t)$ satisfying
\[E\bigg[\bigg(\int_0^T|i_3(t)|dt\bigg)^2
\bigg]=o(\varepsilon^4).\]
Combining (\ref{Ito px_1x_2}), (\ref{Ito Px_1x_1}) and (\ref{Taylor expand f second order 2}) we obtain that
\[d\hat y^{\varepsilon}(t)=-\{\bar f_y(t)\hat y^{\varepsilon}(t)+\bar f_z(t)\hat z^{\varepsilon}(t)+i(t)\}dt+\hat z^{\varepsilon}(t)dW_t,\]
with the residual term $i(t)$ satisfying
\[\mathbb E\bigg[\bigg(\int_0^T|i(t)|dt\bigg)^2
\bigg]=o(\varepsilon^4).\]
\end{proof}

\section{Proof for the Main Results}
\subsection{First Order Condition}

The solution of the linear BSDE (\ref{BSDE First Order}) can be represented via the adjoint SDE. Let $\gamma(t)$ satisfy:
\begin{equation}
\label{adjoint linear SDE}
\left \{
\begin{split}&d\gamma(t)=\bar f_y(t)\gamma(t)dt+\bar f_z(t)\gamma(t)dW_t,\\
&\gamma(0)=1.
\end{split}
\right .
\end{equation}
Applying It\^o's formula to $\gamma(t)y_1(t)$, we shall have
\[y_1(0)=\mathbb E\bigg[\int_0^T\big \{ \gamma(t)(p(t)\delta b(t;u_t)+\delta f(t:u_t))1_{E^{\varepsilon}}(t)\big\}dt\bigg].\]
Choosing $E^{\varepsilon}$ carefully such that $|E^{\varepsilon}|=\varepsilon$ and
\[\mathbb E\bigg[\int_0^T\big \{ \gamma(t)(p(t)\delta b(t;u_t)+\delta f(t:u_t))1_{E^{\varepsilon}}(t)\big\}dt\bigg]
=\varepsilon \mathbb E\bigg[\int_0^T\big \{ \gamma(t)(p(t)\delta b(t;u_t)+\delta f(t:u_t))\big\}dt\bigg].\]
We have
\[J(u^\varepsilon)=y^{\varepsilon}(0)=\bar y(0)+\varepsilon \mathbb E\bigg[\int_0^T\big \{ \gamma(t)(p(t)\delta b(t;u_t)+\delta f(t:u_t))\big\}dt\bigg]+o(\varepsilon).\]
Since $\bar y(0)$ is optimal, we shall have
\[\limsup_{\varepsilon}
\frac{y^{\varepsilon}(0)-\bar y(0)}{\varepsilon} \ge 0,\]
which implies that
\[\mathbb E\bigg[\int_0^T\big \{ \gamma(t)(p(t)\delta b(t;u_t)+\delta f(t:u_t))\big\}dt\bigg] \ge 0,\]
for any $u \in U_{ad}$.
Finally, due to the abitrariness of $u(\cdot)$, we see that \eqref{eq:3.5}
holds. Thus the proof of  Theorem
\ref{thm:1.1} is completed.
\subsection{Second Order Condition}
In this subsection, we are going to prove Theorem \ref{thm:2.2}. Denote by
\[G(t;u):=H_x(t;u)+\bar f_y(t;u)p^*(t)+\bar f_z(t;u)(p^*(t)\bar \sigma_x(t)+q^*(t))\].
Similarly, one can deduce that
\begin{equation}
\label{Seconder Order condition 1}
\mathbb E\bigg[\int_{t_1}^{t_2}\gamma(t)\big \{\delta G(t,v(t))+\delta b^*(t,v(t))P(t)\big \}x_1(t;v(\cdot)) dt\bigg] \ge 0,
\end{equation}
for any $v \in V_{ad}(t_1,t_2)$. Here
\[V_{ad}(t_1,t_2):=\{v \in U_{ad}| v(t) \in V, \text{a.s., a.e. }t \in [t_1,t_2];v(t)=\bar u(t),t\in[0,1]/[t_1,t_2] \}.\]
Note that one can solve (\ref{first oder SDE}) explicitly:
\[x_1(t)=\int_{t_1}^t\Phi(s;t) \delta b(t;v(t))ds,\]
where $\Phi(\cdot;t)$ satisfies
\[d\Phi(s;t)=-\{\Phi(s;t)\bar b_x(s)+\Psi(s;t)\bar \sigma_x(t)\}ds+\Psi(s;t)dW_s,\Phi(t;t)=I.\]
Moreover, for any $t$, $\Phi(s;t)$ is continuous in $s$ almost surely. Thus, we can rewrite the left hand side of (\ref{Seconder Order condition 1}) asb
\[\mathbb E\bigg[\int_{t_1}^{t_2}\int_{t_1}^t\gamma(t)\big\{\delta G(t,v(t))+\delta b^*(t,v(t))P(t)\big\}\Phi(s;t)\delta b(s;v(s))dsdt\bigg].\]
Denote by $\{r_i\}_{i=1}^{\infty}$ the totality of rarional numbers in $[0,1]$, and by $\{u_i\}_{i=1}^{\infty}$ a dense subset of $V$. Since $\mathcal F_t$ is countable generated for $t \in [0,1]$, we can assume that $\{A_{ij}\}_{j=1}^{\infty}$ generates $\mathcal F_{r_i},i=1,2,3,...$. Set
\[Z_{ij}^v(t):=\bar u(t)\chi_{A^c_{ij}}(\omega)\chi_{[0,r_i)}(t)+v\chi_{A_{ij}}(\omega)\chi_{[r_i,1)}(t),\]
for $t\in [0,1],v \in V,i,j=1,2,....$ For each triplet $(i,j,k)$, since
\[\mathbb E\bigg[\gamma(t)(\delta G(t,Z_{ij}^{u_k}(t))+\delta b^*(t,Z_{ij}^{u_k}(t))P(t))\delta b(t,Z_{ij}^k(t))\bigg]\]
is Lebesgue integrable, there is a null subset $T_{ij}^k \subset [0,1]$ such that for $t \in [0,1]/T_{ij}^k$,
\begin{equation*}
\begin{split}
&\lim_{r\rightarrow0+}
\frac{1}{r}\int_{t-r\beta}^{t+r\alpha}
\mathbb E[\gamma(s)(\delta G(s,Z_{ij}^{u_k}(s))+\delta b^*(s,Z_{ij}^{u_k})P(s))\delta b(t,Z_{ij}^{u_s})(s)]ds\\
=&(\alpha+\beta)\mathbb E\bigg[\gamma(t)(\delta G(t,Z_{ij}^{u_k}(t))+\delta b^*(t,Z_{ij}^{u_k})P(t))\delta b(t,Z_{ij}^{u_k})(t)
\bigg]
\end{split}
\end{equation*}
and
\[\lim_{r \rightarrow 0+}\frac{1}{r}\int_{t-r\beta}^{t+r\alpha}
\mathbb E\bigg[(\delta b(s;Z_{ij}^{u_k}(s))-\delta b(t;Z_{ij}^{u_k}(t)))^2\bigg]ds=0.\]
Set
\[T_0:=\cup_{1\le i,j,k \le \infty}T_{ij}^k.\]
Then $T_0$ is a null subset of $[0,1]$. For $t \in [0,1]/T_0 $ and the integers $i$ such that $r_i < t$, consider the perturbed control $v$ as $v(s)=u(s)\chi_{[0,1]/[t-r\beta,t+r\alpha]}(s)+Z_{ij}^k(s)\chi_{[t-r\beta,t+r\alpha]}(s)$. We have
\[\frac{1}{r}\mathbb E\bigg[\int_{t-r\beta}^{t+r\alpha}\int_{t-r\beta}^u\gamma(u)\big\{\delta G(u,Z_{ij}^k(u))+\delta b^*(t,Z_{ij}^k(u))P(u)\big\}\Phi(s;u)\delta b(s;Z_{ij}^k(s))dsdu\bigg]\ge 0.\]
Letting $r$ tend to $0$, we finally get that
\[\mathbb E\bigg[\gamma(t)(\delta G(t;u_k)+\delta b^*(t;u_k)P(t))\delta b(t;u_k)\chi_{A_ij}\bigg]\ge 0.\]
Since $A_{ij}$ generates $\mathcal F_{r_i}$, we have
\[\mathbb E\bigg[\gamma(t)(\delta G(t;u_k)+P(t)\delta b(t;u_k))\delta b(t;u_k)|\mathcal F_{r_i}\bigg] \ge 0,a.s..\]
Since the filtration is generated by the Brownian motion, $\mathcal F_t$ is quasi-left-continuous which implies that all martingales are continuous. Then it holds that
\[\gamma(t)(\delta G(t;u_k)+\delta b^*(t;u_k)P(t))\delta b(t;u_k) \ge 0,a.s..\]
Since $\gamma(t)$ is positive, it is equivalent to
\[(\delta G(t;u_k)+\delta b^*(t;u_k)P(t))\delta b(t;u_k) \ge 0, a.s.\]
From the continuity of the coefficients and the density of $\{u_k\}_{k=1}^{\infty}$, we have
\[(\delta G(t;u)+\delta b^*(t;u)P(t))\delta b(t;u) \ge 0, \forall u \in V, a.s..\]
holds. Therefore we finish the proof of Theorem \ref{thm:2.2}.
\section{Examples}
In this section, we give two examples to illustrate the applications of our second-order maximum principle.
Example 1. The state process of the controlled system is
\begin{equation}
\label{example control system}
\left \{
\begin{split}
&dx(t)=\bigg(\begin{matrix}
-\frac{1}{2}a^2& -u\\
u & -\frac{1}{2}a^2
\end{matrix}\bigg)x(t)dt+\bigg (\begin{matrix}
0 & -a\\
a & 0
\end{matrix}\bigg)x(t)dW(t), 0<t<1,\\
&x(0)=\bigg(\begin{matrix}1\\0\end{matrix}\bigg)
\end{split}
\right .
\end{equation}
with the cost process
\begin{equation}
\label{example cost process}
\left \{
\begin{split}
&dy(t)=-\{\beta y(t)+\gamma z(t)\}dt+z(t)dW(t),\\
&y(T)=\frac{1}{2}|x(T)|^2,
\end{split}
\right .
\end{equation}
where the valued set $U$ of admissible controls is:
\[U=[-1,1].\]
and $a,\beta,\gamma$ are deterministic. For each constant control $u$, equation (\ref{example control system}) can be solved explicitly as
\begin{equation}
\label{example solution}
x(t;u)=\bigg(\begin{matrix}
\cos(ut+aW(t))\\\sin(ut+aW(t))
\end{matrix}\bigg).
\end{equation}
One can check that any admissible control $u(\cdot)$ is optimal in this example. For the admissible reference pair $(x(\cdot),u)$ with $u \in \cal U_{ad}$ being constant, the associated first-order adjoint equation $(p(\cdot;u),q(\cdot;u))$ satisfying the following BSDE:
\begin{equation}
\label{example first order adjoint}
\left \{
\begin{split}
&dp(t)=-\bigg[\bigg(\begin{matrix}
\beta-\frac{1}{2}a^2& \gamma a+u\\
-\gamma a-u &\beta-\frac{1}{2}a^2
\end{matrix}
\bigg)p(t)+\bigg(\begin{matrix}
\gamma & a\\
 -a & \gamma
\end{matrix}
\bigg)q(t)\bigg]dt+q(t)dW_t,\\
&p(1)=x(1).
\end{split}
\right.
\end{equation}
It is solved as
\begin{equation*}
\left \{
\begin{split}
&p(t;u)=\exp(\beta(T-t))\bigg(
\begin{matrix}
\cos(ut+aW_t)\\sin(ut+aW_t)
\end{matrix}
\bigg)\\
&q(t;u)=\exp(\beta(T-t))\bigg(
\begin{matrix}
-a\sin(ut+aW_t)\\a\cos(ut+aW_t)
\end{matrix}
\bigg).
\end{split}
\right .
\end{equation*}
Thus the Hamiltonian can be calculated which shows that  $H(t,x(t;u),y(t;u),z(t;u),v,p(t;u),q(t;u))$ is independent of $v$. Hence any constant control $u$ is singalar on $U$. Consider the second order adjoint equation:
\begin{equation}
\left \{
\begin{split}
&dP(t)=-[\beta P(t)+(f_x+\gamma\sigma_x)^*P(t)+P(t)(f_x+\gamma\sigma_x)+\sigma_x^*P(t)\sigma_x+\gamma Q(t)+\sigma_x^*Q(t)+Q(t)\sigma_x]dt+Q(t)dW_t,\\
&P(1)=I
\end{split}
\right .
\end{equation}
with
\[f_x=\bigg(\begin{matrix}
-\frac{1}{2}a^2& -u\\ u & -\frac{1}{2}a^2
\end{matrix}\bigg),\sigma_x=\bigg(\begin{matrix}
0& -a\\ a & 0
\end{matrix}\bigg).\]
Obviously, $P(t)=\exp(\beta(T-t))I,Q(t)\equiv 0$. Then we have
\[\delta G(t;v)=-\exp(\beta(T-t))(v-u)^2,\delta b^*(t;v)P(t)\delta b(t;v)=\exp(\beta(T-t))(v-u)^2.\]
It implies that any constant control $u$ satisfies our second-order maximum principle. This show that the second term in can not be crossed out in (\ref{second order condition}).\\

Example 2. The control system is
\begin{equation}
\left \{
\begin{split}
&dx(t)=u(t)dt+(x-1)dW_t,u_t\in U:=\{-1,0,1\}\\
&x(0)=1
\end{split}
\right .
\end{equation}
and the cost process is defined as
\begin{equation*}
\left \{
\begin{split}
&dy(t)=-f(y(t),z(t))dt+z(t)dW(t),\\
&y(1)=\pm\frac{1}{2}(x(1)-1)^2,
\end{split}
\right .
\end{equation*}
with $f$ be any deterministic function. For both cost functionals, the constant control $u \equiv 0$ is singular on $U$ since the corrsponding first-order adjoint processes are identically zero. The second adjoint processes are $(P(t),Q(t)\equiv 0)$ with $P(t)$ solves the following ODE:
\[dP(t)=-[\bar f_y(t)+2\bar f_z(t)+1]P(t)dt,P(1)=\pm \frac{1}{2}.\]
From Thoerem \ref{thm:2.2}, we see that $u \equiv 0$ is a candidate for optimal controls at the case $y(1)=\frac{1}{2}(x(1)-1)^2$, and necessarily not an optimal control at the other case.


\begin{thebibliography}{99}
\bibitem{Be75}
Bell, D. J., \& Jacobson, D. H. (1975). Singular optimal control problems (Vol. 117). Elsevier.



\bibitem{Be}
Bensoussan, A. (1982). Lectures on stochastic control. In Nonlinear filtering and stochastic control (pp. 1-62). Springer, Berlin, Heidelberg.

\bibitem{Bis}
Bismut, J. M. (1978). An introductory approach to duality in optimal stochastic control. {\it SIAM review,} 20(1), 62-78.

\bibitem{Bo12}
Bonnans, J. F., \& Silva, F. J. (2012). First and second order necessary conditions for stochastic optimal control problems. {\it Applied Mathematics \& Optimization,} 65(3), 403-439.


\bibitem{Do99}
Dokuchaev, N., \& Zhou, X. Y. (1999). Stochastic controls with terminal contingent conditions.
{\it Journal of Mathematical Analysis and Applications,} 238(1), 143-165.


\bibitem{Fr17}
Frankowska, H., Zhang, H., \& Zhang, X. (2017). First and second order necessary conditions for stochastic optimal controls.
{\it Journal of Differential Equations,} 262(6), 3689-3736.


\bibitem{GK}
Gabasov, R., \& Kirillova, F. M. (1972). High order necessary conditions for optimality. {\it SIAM Journal on Control,} 10(1), 127-168

\bibitem{Gi93}
Gift, S. J. G. (1993). Second-order optimality principle for singular optimal control problems.
{\it Journal of optimization theory and applications,} 76(3), 477-484.


\bibitem{Ha}
Haussmann, U. G. (1976). General necessary conditions for optimal control of stochastic systems. In Stochastic Systems: Modeling, Identification and Optimization, II (pp. 30-48). Springer, Berlin, Heidelberg.




\bibitem{Hu17}
Hu, M. (2017). Stochastic global maximum principle for optimization with recursive utilities.
{\it Probability, Uncertainty and Quantitative Risk,} 2(1), 1.




\bibitem{Ji06}
Ji, S., \& Zhou, X. Y. (2006). A maximum principle for stochastic optimal control with terminal state constraints, and its applications.
{\it Communications in Information \& Systems,} 6(4), 321-338.

\bibitem{Ka84}
Kazemi-Dehkordi, M. A. (1984). Necessary conditions for optimality of singular controls.
{\it Journal of optimization theory and applications,} 43(4), 629-637.

\bibitem{Kr77}
Krener, A. J. (1977). The high order maximal principle and its application to singular extremals.
 {\it SIAM Journal on Control and Optimization,} 15(2), 256-293.

\bibitem{Ku}
Kushner, H. J. (1972). Necessary conditions for continuous parameter stochastic optimization problems.
{\it SIAM Journal on Control,} 10(3), 550-565.
\bibitem{Lou}
Lou, H. \& Yong, J.(2017) Second-Order Necessary Conditions for Optimal Control of Semilinear Elliptic Equations with Leading Term Containing Controls. {\it arXiv:1703.08649}
	

\bibitem{Pe}Peng, S. (1990). A general stochastic maximum principle for optimal control problems.
     {\it SIAM Journal on control and
     optimization,} 28(4), 966-979.

\bibitem{Pe93}Peng, S. (1993). Backward stochastic differential equations and applications to optimal control. {\it Applied Mathematics and
     Optimization,} 27(2), 125-144.

\bibitem{Pe98}
Peng, S.(1998). Open problems on backward stochastic differential equations. In: Chen, S, Li, X, Yong, J, Zhou,
XY (eds.) Control of distributed parameter and stocastic systems, pp. 265¨C273, Boston: Kluwer Acad.
Pub.


\bibitem{Sh06}Shi, J., \& Wu, Z. (2006). The Maximum I Principle for Fully Coupled Forward-backward Stochastic Control System.
     {\it Acta Automatica Sinica,} 32(2), 161.
    \bibitem{Ta10}
   Tang, S. (2010). A second-order maximum principle for singular optimal stochastic controls.
    {\it Discrete Contin. Dyn. Syst. Ser. B,} 14, 1581-1599.


  \bibitem {Wu98}
    Wu, Z.(1998). Maximum principle for optimal control problem of fully coupled forward-backward stochastic
systems. {\it Syst. Sci. Math. Sci.} 11, 249¨C
  \bibitem{Wu13}
  Wu, Z. (2013). A general maximum principle for optimal control of forward¨Cbackward stochastic systems.
   {\it Automatica,} 49(5), 1473-1480.

\bibitem{Xu95}
Xu, W. (1995). Stochastic maximum principle for optimal control problem of forward and backward system. The ANZIAM Journal, 37(2), 172-185.
\bibitem{Yo10}
Yong, J. (2010). Optimality variational principle for controlled forward-backward stochastic differential equations with mixed initial-terminal conditions.
{\it SIAM Journal on Control and Optimization,} 48(6), 4119-4156.

\bibitem{Zh15}
Zhang, H., \& Zhang, X. (2015). Pointwise second-order necessary conditions for stochastic optimal controls, Part I: The case of convex control constraint.
{\it SIAM Journal on Control and Optimization,} 53(4), 2267-2296.

\bibitem{Zh17}
Zhang, H., \& Zhang, X. (2017). Pointwise second-order necessary conditions for stochastic optimal controls, Part II: The general case.
{\it SIAM Journal on Control and Optimization,} 55(5), 2841-2875.

\bibitem{Zh96}
Zhou, Q. (1996). Second-order optimality principle for singular optimal control problems.
{\it Journal of optimization theory and applications,} 88(1), 247-249.


\end{thebibliography}
\end{document}